\documentclass[12pt,reqno]{amsart}
\usepackage{pgfplots}
\usepackage{tikz}
\pgfplotsset{compat=newest}
\pgfplotsset{ytick style={draw=none}}
\pgfplotsset{xtick style={draw=none}}
\usepackage{amsmath,amsfonts,amssymb,amsthm}
\usepackage{graphicx}
\graphicspath{ }
\usepackage{mathrsfs}
\usepackage{epstopdf}
\usepackage{csquotes}
\usepackage{wrapfig}
\usepackage{accents}
\usepackage{caption}
\usepackage{subcaption}
\usepackage{calligra}
\usepackage[colorlinks]{hyperref}
\usepackage{kantlipsum}
\usepackage{cleveref}
\usepackage{pifont}
\hypersetup{colorlinks, 
	breaklinks,
	linkcolor=red,
	citecolor=red,
	linktocpage=true}

\numberwithin{figure}{section}

\theoremstyle{plain}
\newtheorem{thm}{Theorem}[section]
\newtheorem{lem}[thm]{Lemma}
\newtheorem{prop}[thm]{Proposition}

\theoremstyle{definition}
\newtheorem{defn}{Definition}[section]
\newtheorem{conj}{Conjecture}[section]
\newtheorem{exmp}{Example}[section]

\theoremstyle{remark}
\newtheorem{rem}{Remark}

\usepackage{mathtools}
\makeatletter
\@namedef{subjclassname@2020}{%
	\textup{2020} Mathematics Subject Classification}
\makeatother
\author[M. A. Sarkar]{MABUD ALI SARKAR}
\address{Department of Mathematics\\ The University of Burdwan \\ Burdwan-713101, India.}
\email{mabudji@gmail.com}
\author[A. A. Shaikh]{Absos Ali Shaikh}
\address{Department of Mathematics\\ The University of Burdwan \\ Burdwan-713101, India.}
\email{aashaikh@math.buruniv.ac.in}
\title{Rigidity and unlikely intersection for stable $p$-adic dynamical systems}

\begin{document}

	\begin{abstract}
		Berger asked the question \enquote{To what extent the preperiodic points of a stable $p$-adic power series determines a stable $p$-adic dynamical system} ? In this work we have applied the preperiodic points of a stable $p$-adic power series in order to determine the corresponding stable $p$-adic dynamical system. 
	\end{abstract}
	\maketitle

	\section{Introduction and motivation}
	Let $K$ be the finite extension of the $p$-adic field $ \mathbb{Q}_p$ with ring of integers $\mathcal{O}_K$, and the unique maximal ideal $\mathfrak{m}_K$. We denote the units in $\mathcal{O}_K$ by $\mathcal{O}_K^{*}$. Let $\bar K$ be the algebraic closure of $K$ and $\bar{\mathfrak{m}}_K$ be the integral closure of $\mathfrak{m}_K$ in $\bar{K}$. Let $C_p$ be the $p$-adic completion of $\bar{K}$ and denote $\mathfrak{m}_{C_p}=\{z \in C_p~|~ |z|_p<1 \}$.
	
	In \cite{LB1}, Berger studied to what extent the torsion points $\text{Tors}(F)$ of a formal group $F$ over $\mathcal{O}_K$ determines the formal group. He proved that if $\text{Tors}(F_1) \cap \text{Tors}(F_2)$ is infinite then $F_1=F_2$. He further asked the question, if $\mathcal{D}$ is a stable $p$-adic dynamical system, then: 
	
	\hspace*{2cm}\enquote{To what extent the preperiodic points $\text{Preper}(\mathcal{D})$ determines $\mathcal{D}$ ?}

	In this work, we have answered this question by proving our main Theorem \ref{t3.4} in Section \ref{s3}. We have also provided an alternate proof of it following some examples in Section \ref{s4}. The proofs relies on the following tools:
	\begin{enumerate}
		\item[(a)] The first proof uses the correspondence between $\text{Tors}(F)$ and $\text{Preper}(\mathcal{D})$.
		\item[(b)] 
		while the alternate proof uses the following two facts:
		\begin{enumerate}
			\item[(i)] Galois correspondence of a stable $p$-adic dynamical system $\mathcal{D}$. Indeed, we proved that given any stable $p$-adic dynamical system $\mathcal{D}$ over $\mathcal{O}_K$, there exists a $\sigma \in \text{Gal}(\bar{K}/K)$ and a stable series $w(x) \in \mathcal{D}$ such that   
			$	\sigma(x)=w(x),\ \forall \ x \in \text{Preper}(\mathcal{D}).$
			\item[(ii)] Rigidity of power series on open unit disk $\mathfrak{m}_{C_p}$. We say that a subset $\mathcal{Z} \subset \mathfrak{m}_{C_p}$ is Zariski dense in $\mathfrak{m}_{C_p}$ if every power series $h(x) \in \mathcal{O}_K[[x]]$ that vanishes on $\mathcal{Z}$ is necessarily equal to zero. A subset $\mathcal{Z} \subset \mathfrak{m}_{C_p}$ is Zariski dense in $\mathfrak{m}_{C_p}$ if and only if it is infinite.
		\end{enumerate}
	\end{enumerate}

	\section{$p$-adic dynamical system and some results}
	In this section, we recall some prelimineries, and prove some helpful results:
	\begin{defn} \cite{Sil}
		A (discrete) \textit{dynamical system} consists of a set $\Gamma$ and a function $\gamma: \Gamma \to \Gamma$. Its dynamics is indeed the study of the behavior of the points in $\Gamma$ by repeatedly applying $\gamma$ on the points of $\Gamma$ i.e., we study the iterates of $\gamma$. If we consider the $n^{\text{th}}$ iterate $$\gamma^{\circ n}(x)=\underbrace{\gamma \circ \gamma \circ \cdots \circ \gamma(x)}_{\text{n iterates}},$$
		then the orbit of $x \in \Gamma$ is defined by $O_{\gamma}(x)=\left\{x,~\gamma(x),~\gamma^{\circ 2}(x),~\gamma^{\circ 3}(x), \cdots \right\}.$
		\begin{enumerate}
			\item[(i)]The point $x$ is called \textit{periodic} of period $n$ if $\gamma^{\circ n}(x)=x$ for some $n \geq 1$.
			\item[(ii)] If $\gamma(x)=x$, then $x$ is a fixed point.
			\item[(iii)] A point $x$ is \textit{preperiodic} if some iterate $\gamma^{\circ i}(x)$ is periodic i.e., there exists $m,n$ such that $\gamma^{\circ m}(x)=\gamma^{\circ n}(x)$. In other words, $x$ is preperiodic if its orbit $O_{\gamma}(x)$ is finite.  
		\end{enumerate}
	\end{defn}
	
	\begin{defn} \cite{LB1}
		A stable $p$-adic dynamical system $\mathcal{D}$ over $\mathcal{O}_K$ is a collection of $p$-adic power series in $\mathcal{O}_K[[x]]$ without constant term such that the power series commutes with each other under formal composition. A power series $f$ in $\mathcal{D}$ is called stable if $f'(0)$ is neither $0$ nor a root of $1$.
		We say that $\mathcal{D} \subseteq x \cdot \mathcal{O}_K[[x]]$ is a stable $p$-adic dynamical system of finite height if the elements of $\mathcal{D}$ commute with each other under composition, and if $\mathcal{D}$ contains a stable series $f$ such that $f'(0) \in \mathfrak{m}_K$ and $f(x) \not\equiv 0$ mod $\mathfrak{m}_K$ (i.e., $f$ is of finite height) as well as a stable series $u$ such that $u'(0) \in \mathcal{O}_K^{\times}$.  The collection $\mathcal{D}$ can be made as large as possible in the sense that whenever a stable power series commutes with any member of $\mathcal{D}$, it belongs to $\mathcal{D}$. 
		Such a collection $\mathcal{D}$ is the main object in $p$-adic dynamical systems \cite{JL}.
	\end{defn}
	\begin{exmp} \label{e2.1}
		If $F$ is a formal group law of finite height over $\mathcal{O}_K$, then the endomorphism ring $\text{End}_{\mathcal{O}_K}(F)$ of $F$ is a stable $p$-adic dynamical system. 
	\end{exmp} 
	
	\begin{prop} \label{p2.1}
		For an invertible power series, preperiodic points are exactly the periodic points, i.e., fixed points of iterates of $u$.   
	\end{prop}
	\begin{proof}
		Let $u(x) \in x \cdot \mathcal{O}_K[[x]]$ be invertible. For any preperiodic point $\alpha$ of $u(x)$, there exists natural numbers $m,n$ with $m>n$ such that $u^{\circ m }(\alpha)=u^{\circ n }(\alpha)$. Since $u(x)$ is invertible, $u^{\circ (-n)}(x)$ exists in $x \cdot \mathcal{O}_K[[x]]$ and hence $u^{\circ (m-n)}(\alpha)=\alpha$. 
	\end{proof}
	If $u$ is an invertible series over $\mathcal{O}_K$, then the preperiodic points of $u$ are exactly the periodic points by Proposition \ref{p2.1}. Now we remember that: the only periodic points of $u$ are roots of $u^{\circ p^m}(x)-x$ for some $m \in \mathbb{N}$. The full ring $\mathbb{Z}_p$ acts on the invertible members of the dynamical system $\mathcal{D}$. For, the series $u^{\circ p^m}$ converge to the identity in the appropriate topology, and thus the map $\mathbb{Z} \to \mathcal{D}$ by $n \mapsto u^{\circ n}$ is continuous when $\mathbb{Z}$ has the $p$-adic topology, so extends to $\mathbb{Z}_p \to \mathcal{D}$. It follows from this that if $ m \in \mathbb{Z}$ and $m=p^r n$ with $p \nmid n$, then the fixed points of $u^{\circ m}$ are the fixed points of $u^{\circ p^r}$. To be more precise, we consider the following two lemmas:

	\begin{lem}
		Let $u$ be an invertible series in $\mathcal{O}_K[[x]]$. Then for every natural number $n \geq 0$, for any $ \lambda \in \bar{K}$ with $v(\lambda)>0$, if $\lambda$ is a fixed point of $u$, then $\lambda$ is also a fixed point of $u^{\circ p^n}$.  
	\end{lem}   
	\begin{proof}
		Note that $u^{\circ 2}(\lambda)=u(u(\lambda))=u(\lambda)=\lambda$.
		Thus by induction on $n$, the result follows.
	\end{proof}
	
	\begin{lem}
		Let $u$ be an invertible series in $\mathcal{O}_K[[x]]$. Then for every natural number $n \geq 0$, for any $ \lambda \in \bar{K}$ with $v(\lambda)>0$, if $\lambda$ is a fixed point of $u$, then $\lambda$ is also a fixed point of $u^{\circ p^n}$. More generally, for $ z \in \mathbb{Z}_p$, $\lambda$ is also a fixed point of $u^{\circ z}$.
	\end{lem}
	\begin{proof}
		We recall that the map $\mathbb{Z} \to \mathcal{O}_K[[x]]$ by $n \mapsto u^{\circ n}$ is continuous when $\mathbb{Z}$ has the $p$-adic topology and $\mathcal{O}_K[[x]]$ has $(\mathfrak{m}_K,~x)$-adic topology. This latter topology also has the property that if $\{U_i\},~U$ are invertible series in $\mathcal{O}_K[[x]]$ with limit $\lim_i U_i=U$,then $\lim_iU_i(\lambda)=U(\lambda)$. That is, evaluation at $\lambda$ is a continuous map from $\mathcal{O}_K[[x]]$ to $\{\xi \in \bar{K}: v(\xi)>0 \}$.
		
		Now suppose that $\lambda$ is a fixed point of $u$, and $z \in \mathbb{Z}_p$. There is a sequence of positive integers $\{z_i\}$ with limit $z$, and so $\lambda$ is a fxed point of each $u^{\circ z_i}$, so that $u^{\circ z}(\lambda)=u^{\circ \lim_i z_i}(\lambda)=\lim_i \left(u^{\circ z_i}(\lambda) \right)=\lim_i \lambda=\lambda$. 
	\end{proof}

	We define the following two sets 
	\begin{eqnarray*}
		\text{Preper}(u)&=\bigcup_{n}\{x \in \mathcal{O}_K | u^{\circ p^n}(x)=x  \} =\text{all preperiodic points of an invertible series $u \in \mathcal{D}$} \\
		T(f)&=\bigcup_{n}\{x \in \bar{\mathfrak{m}}_K | f^{\circ n}(x)=0  \}
		=\text{all torsion points of an noninvertible series $f \in \mathcal{D}$}.
	\end{eqnarray*}
	We note the following interesting result, which says that $\text{Preper}(\mathcal{D})$ is independent of choices of stable series in $\mathcal{D}$:
	\begin{prop} \cite{JL} \label{p2.4}
		Let $u, f \in \mathcal{D}$ be invertible and noninvertible series respectively. Then the set of roots of iterates of $f$ is equal to the set periodic points of $u(x)$. That is, if $T(f)$ denotes the set of roots of iterates of $f$, then $T(f)=\text{Preper}(u).$
	\end{prop}

	\section{The Main Results} \label{s3}
	
	We	start with a Conjecture:	
	\begin{conj}\label{c3.1} \cite{LB2}
		If $f$ and $u$ are, respectively, two stable noninvertible and invertible power series in a stable $p$-adic dynamical system $\mathcal{D}$, then there exists a formal group $F$ with coefficients in $\mathcal{O}_K$, two endomorphisms $f_F$ and $u_F$ of $F$, and a nonzero power series $h$ such that $f \circ h=h \circ f_F$ and $u \circ h=h \circ u_F$. We call $h$ to be the isogeny from $f_F$ to $f$. 
	\end{conj}
	\begin{rem}
		The conjecture \ref{c3.1} is proved in \cite[Theorem.~B]{LB2} for $K=\mathbb{Q}_p$. This conjecture resembles to that one given by Lubin in \cite{JL} while \cite{HCL1}, \cite{HCL2} and \cite{HCL3} proved several results in the support of Lubin's  conjecture, which says, if a noninvertible series commutes with an invertible series, there is a formal group somehow in the background. 
	\end{rem}

	For the above formal group $F$ over $\mathcal{O}_K$, its endomorphism ring $\text{End}_{\mathcal{O}_K}(F)$ is a stable $p$-adic dynamical system. We denote by $\text{Tors}(F)=\bigcup_n T(n,f_F)$, the torsion points of $F$, where $T(n,f_F)=\{\alpha \in \bar{\mathfrak{m}}_K: f_F^{\circ n}(\alpha)=0\}$. Then we have the following nice result:
	\begin{thm}\cite{LB1}\label{t3.1}
		If $F_1$ and $F_2$ are two formal groups over $\mathcal{O}_K$ and if $\text{Tors}(F_1)\cap \text{Tors}(F_2)$=infinite, then $F_1=F_2$. 
	\end{thm}
	\begin{defn} 
		Let $f(x)$ and $g(x)$ be two noninvertible stable power series over $\mathcal{O}_K$ without constant term. We call a power series $h(x) \in \mathcal{O}_K[[x]]$ an $\mathcal{O}_K$-isogeny of $f(x)$ into $g(x)$ if $h \circ f=g \circ h$. If $u(x)$ be any invertible series in $\mathcal{O}_K[[x]]$ then $u \circ h$ is also an $\mathcal{O}_K$-isogeny of $f$.
	\end{defn}
	
	Next we prove the following lemma:
	\begin{lem} \label{l3.2} 
		Let $f(x)$ and $g(x)$ be noninvertible stable power series over $\mathcal{O}_K$ with finite Weierstrass degree. Let $h$ be an isogeny of $f$ into $g$, then $h$ maps $T(f)$ into $T(g)$. Moreover, $h: T(f) \to T(g)$ is surjective.
	\end{lem}
	\begin{proof}
		At first we will show that $h(0)=0$. Since $g(h(0))=h(f(0))=h(0)$, $h(0)$ is a fixed point of $g(x)$. But $g(x)$ being noninvertible can have $0$ as its only fixed point and hence $h(0)=0$.

		Now let $\alpha \in T_n(f) \subset T(f)$, then $g^{\circ n}(h(\alpha))=h(f^{\circ n}(\alpha))=h(0)=0$. This implies $h(\alpha) \in T(g)$. This shows that $h$ maps $T(f)$ to $T(g)$.
		
		On the other hand, take any $\beta \in T_m(g) \subset T(g)$ for some natural number $m \in \mathbb{N}$ and let $\alpha \in \bar{\mathfrak{m}}_K$ such that $h(\alpha)=\beta$. We need to show that $\alpha \in T(f)$. For, $h(f^{\circ m}(\alpha))=g^{\circ m}(h(\alpha))=0$ implies $f^{\circ m}(\alpha)$ is a root of $h(x)$ which is also true for all $n \geq m$. Since $h(x)$ can have only finitely many roots in $\bar{\mathfrak{m}}_K$, we must have $$f^{\circ n+\tilde{n}}(\alpha)=f^{\circ n}(\alpha)
		\ \text{for some} \ n,~\tilde{n} \in \mathbb{N}.$$
		This implies that $f^{\circ n}(\alpha)$ is a fixed point of $f^{\circ \tilde{n}}(x)$. Since $f^{\circ \tilde{n}}(x)$ is noninvertible, it has the only fixed point $0$ and hence $f^{\circ n}(\alpha)=0$. Thus $\alpha \in T_n(f) \subset T(f)$. Thus $h$ is surjective.
	\end{proof}
	\begin{defn}
		We denote a stable $p$-adic dynamical system $\mathcal{D}$ by the package $(\mathcal{D},f,u;F,f_F,u_F;h)$, where $F$ is the background formal group of $\mathcal{D}$  with  $f_F,u_F$ noninvertible and invertible endomorphisms respectively, while $u, f$ are invertible and noninvertible power series in $\mathcal{D}$ respectively, along with an isogeny map $h:f_F \to f$  as in Conjecture \ref{c3.1}
	\end{defn}
	
	Now we will prove the uniqueness of the formal group $F$ in Conjecture \ref{c3.1}:
	\begin{prop} \label{p3.4}
		There exists a unique formal group $F$ for each stable $p$-adic dynamical system $\mathcal{D}$ in the Conjecture \ref{c3.1}.
	\end{prop}
	\begin{proof}
		Let $\mathcal{D}$ be a stable $p$-adic dynamical system over $\mathcal{O}_K$ consisting of a noninvertible series $f$ and an invertible series $u$. 
		By Conjecture \ref{c3.1}, there exists a formal group $F$ over $\mathcal{O}_K$ with endomorphisms $f_F$, $u_F$ and an isogeny $h$ from $f_F$ to $f$.
		We want to show that that $F$ is unique. If possible let there exists another formal group $G$ over $\mathcal{O}_K$ with endomorphisms $f_G$, $u_G$ and an isogeny, say, $h'$ from $f_G$ to $f$.
		By Lemma \ref{l3.2}, we have the surjections $h:T(f_F) \to T(f)$ and $h':T(f_G) \to T(f)$. Therefore for every $\alpha \in \text{Preper}(\mathcal{D})$ there exists some $\beta_1 \in \text{Tors}(F)$ and some $\beta_2 \in \text{Tors}(G)$ such that $h(\beta_1)=\alpha=h'(\beta_2)$. This shows that both $\text{Tors}(F)$ and $\text{Tors}(G)$ has infinitely many points in common and thus by the Theorem \ref{t3.1}, we get $F=G$.
	\end{proof}

	We will now proof the main result of the paper:
	\begin{thm}\label{t3.4}
		If $(\mathcal{D}_1,f_1,u_1;F_1,f_{F_1},u_{F_1};h_1)$ and $(\mathcal{D}_2,f_2,u_2;F_2,f_{F_2},u_{F_2};h_2)$ are two dynamical systems over $\mathcal{O}_K$ such that $ \text{Preper}(\mathcal{D}_1) \cap \text{Preper}(\mathcal{D}_2)$ is infinite , then $\mathcal{D}_1=\mathcal{D}_2$.
	\end{thm}
	\begin{proof}
		By Lemma \ref{l3.2}, the isogenies $h_i$ defines surjective maps $h_i:T(f_{F_i}) \to T(f_i), \ i=1,2$.
		Thus for any $\beta_i \in T(f_i)$ there exists an $\alpha_i \in T(f_{F_i})$ such that $h_i(\alpha_i)=\beta_i$. We note that  $\text{Tors}(F_1) \cap \text{Tors}(F_2)$ will have infinitely many points in common if $\text{Preper}(\mathcal{D}_1) \cap \text{Preper}(\mathcal{D}_2)$=infinite, because the isogenies $h_i$ maps $T(f_{F_i})$ into $T(f_i)$ by Lemma \ref{l3.2}. But given that $\text{Preper}(\mathcal{D}_1) \cap \text{Preper}(\mathcal{D}_2)$ is infinite, and hence $\text{Tors}(F_1) \cap \text{Tors}(F_2)$ is infinite. Therefore by Theorem \ref{t3.1}, we conclude $F_1=F_2.$
		Hence by the uniqueness property of Proposition \ref{p3.4}, we must have $\mathcal{D}_1=\mathcal{D}_2$.
	\end{proof}
	\section{Alternative proof of Theorem~\ref{t3.4}} \label{s4}
	In this section we give another proof of the main Theorem \ref{t3.4} which deserved to be included because of its beauty. We are indebted to the ideas of \cite{LB1}. At first, we note the following beautiful result:
	\begin{thm} \cite{LB1}\label{t4.1}
		Given a formal group $F$ over $\mathcal{O}_K$ with torsion points $\text{Tors}(F)$, there is a stable endomorphism $u_F$ of $F$ and $\sigma \in \text{Gal}(\bar{K}/K)$ such that 
		$$\sigma(z)=u_F(z) \ \text{for all} \ z \in \text{Tors}(F).$$ 
	\end{thm}
	Now we prove a similar result for a stable $p$-adic dynamical system:
	\begin{thm}\label{t4.2}
		Let $\mathcal{D}$ be a stable $p$-adic dynamical system, then there exists a stable power series $w(x) \in \mathcal{D}$ and an $\sigma \in \text{Gal}(\bar{K}/K)$ such that \begin{align}
			\label{e0} \sigma(z)=w(z), \ \text{for all} \ z \in \text{Preper}(\mathcal{D}).
		\end{align} 
	\end{thm}
	\begin{proof} By the Conjecture \ref{c3.1}, if $f$ and $u$ are two stable noninvertible and invertible power series in $\mathcal{D}$, then there exists a formal group $F$ with coefficients in $\mathcal{O}_K$, two endomorphisms $f_F$ and $u_F$ of $F$, and a nonzero power series $h$ such that $f \circ h=h \circ f_F$ and $u \circ h=h \circ u_F$, where $h$ is the isogeny from $f_F$ to $f$.

		By Lemma \ref{l3.2}, $h$ maps $T(f_F)$ into $T(f)$, and hence for every $\beta \in \text{Tors}(F)$, we get $h(\beta) \in \text{Preper}(\mathcal{D})$. Moreover, by Lemma \ref{l3.2}, we see $h: T(f_F) \to T(f)$ is also surjective. Thus for every $\alpha \in \text{Preper}(\mathcal{D})$ there exists some $\beta \in \text{Tors}(F)$ such that $h(\beta)=\alpha$.
		From the Theorem \ref{t3.1}, we have
		\begin{equation} \label{e2}
		\sigma(z)=u_F(z) \ \text{ for all} \ z \in \text{Tors}(F).
		\end{equation} 
		Now it remains to show that we can replace $u_F$ by an element $w \in \mathcal{D}$ in equation \eqref{e2} such that $w \notin \text{End}_{\mathcal{O}_K}(F)$.  
		Applying the isogeny $h$ both sides of equation \eqref{e2} and using the relation $u \circ h=h \circ u_F$ from Conjecture \ref{c3.1}, we get 
		\begin{align}
			\nonumber\sigma(z)&=u_F(z) \ \text{ for all} \ z \in \text{Tors}(F), \\
			\nonumber	\Rightarrow	h(\sigma(z))&=(h \circ u_F)(z)  \ \text{ for all} \ z \in \text{Tors}(F),  \\
			\label{e3}	\Rightarrow	\sigma(h(z))&= u(h(z) \ \text{ for all} \ z \in \text{Tors}(F), \ (\because \sigma(h(z))=h(\sigma(z)) \\
			\label{e4}	\Rightarrow \sigma(\tilde{z})&=u(\tilde{z}), \ \text{ for all} \ \tilde{z}=h(z) \in \text{Preper}(\mathcal{D}).
		\end{align}
		The relation \eqref{e4} follows from the relation \eqref{e3} because $h$ maps $T(f_F)$ into $T(f)$, by Lemma \ref{l3.2}.
		Finally denoting $w(x):=u(x) \in \text{Preper}(\mathcal{D})$, we are done.
	\end{proof}
	The following example describes a situation when we get a relation like \eqref{e0}:
	\begin{exmp}
		Let $f(x) \in x \cdot \mathcal{O}_K[[x]]$ be a noninvertible and irreducible polynomial of degree 5 with set of zeros  $\Theta=\{r_1,r_2,r_3,r_4, r_5\}$ such that the extension $K(\Theta):=K(r_1,r_2,r_3,r_4,r_5)$ is Galois with Galois group say, $\text{Gal}(K(\Theta)/K)$. Any $\tau \in \text{Gal}(K(\Theta)/K)$ permutes the elements of $\Theta$. Define some $w(x) \in x \cdot \mathcal{O}_K[[x]]$ by $w(x)=x+s(x)f(x)$ for some $s(x) \in \mathcal{O}_K[[x]]$. We claim there exist some $\tau \in \text{Gal}(K(\Theta)/K)$ so that $\tau(r_i)=w(r_i)$ for all $r_i \in \Theta$. \\
		\textbf{Case I:} Suppose $w(x)$ fixes one of $r_i$, then $g(x)-x$ has root $r_i$, and so $f(x)~|~(w(x)-x)$. In this case $w(r_i)=r_i$ for every $i=1,2,3,4,5$, which implies $w(x)$ induces the identity permutaion on the set $\Theta$, that is, for $\tau=\text{Id} \in \text{Gal}(K(\Theta)/K)$ we have $\tau(z)=w(z)$ for all $z \in \Theta$. \\
		\textbf{Case II:} Suppose $w(x)$ do not fix any of $r_i, i=1,2,3,4,5$. Since the splitting field of $f(x)$ is of degree 5, either $w$ induces a permutation $(r_1r_2r_3r_4r_5)$ or a permutation $(r_1r_2r_3)(r_4r_5)$. If $w$ induces the permutation $ (r_1r_2r_3)(r_4r_5)$, then $w^{\circ 2}$ induces the permutation of type $(r_1r_2r_3)(r_4)(r_5)$. This shows $r_4$ and $r_5$ are the fixed points and the permutation is not identity. So by the argument of Case I, this can not happen. Hence $w$ induces the $5$-cycle $(r_1r_2r_3r_4r_5)$. Therefore by repeated composition of $w$ each $r_1,r_2,r_3,r_4,r_5$ can be expressed as a polynomial in $r_1$. In other words, the splitting field is $K(\Theta)=K[r_1]$ of degree $5$.  Now we claim that $w$ induces the same permutation as a power of $\tau$. Without loss of genarality, choose the notation such that $\tau=(r_1r_2r_3r_4r_5)$. Now we have the following subcases: \\ $(i)$ if $w(r_1)=r_2$ then $\tau(w(r_1))=\tau(r_2) \Rightarrow w(\tau(r_1))=r_3 \Rightarrow w(r_2)=r_3$. Applying $\tau$ on both sides of $w(r_2)=r_3$, we get $w(r_3)=r_4$. Once again, applying $\tau$ on $w(r_3)=r_4$, we get $w(r_4)=r_5$. So indeed $w$ induces $\tau$. \\ $(ii)$ if $w(r_1)=r_3$, similarly, $w$ induces $\tau^2$. \\
		$(iii)$ if $w(r_1)=r_4$, similarly, $w$ induces $\tau^3$. \\
		$(iv)$ if $w(r_1)=r_5$, similarly, $w$ induces $\tau^4$. \\
		Finally, since there is a continuous surjection $\operatorname{Gal}(\bar K/K) \twoheadrightarrow \operatorname{Gal}(K(\Theta)/K)$, for the given $\tau$ there exist a $\sigma \in \text{Gal}(\bar K/K)$ such that $\sigma\rvert_{K(\Theta)}=\tau$ so that $\sigma|_{K(\Theta)}(r_i)=w(r_i)$ for all $r_i \in \Theta$.
	\end{exmp}  
	\begin{lem} \label{l4.3}
		Let $\mathcal{D}$ be stable $p$-adic dynamical system over $\mathcal{O}_K$ and $I(x) \in x \cdot \mathcal{O}_K[[x]]$. If $I(z) \in \text{Preper}(\mathcal{D})$ for infinitely many $z$, then $I \in \mathcal{D}$.
	\end{lem}
	\begin{proof}
		Since $h$ maps $T(f_F)$ into $T(f)$ by Lemma \ref{l3.2}, Theorem \ref{t4.2} implies there is a $\sigma \in \text{Gal}(\bar K/K)$ and a $w \in \mathcal{D}$ such that $\sigma(z)=w(z)~\forall~z \in \text{Preper}(\mathcal{D})$. If $z \in \text{Preper}(\mathcal{D})$, then we have 
		\begin{align} \label{e5}
			\sigma(I(z))=w(I(z)).
		\end{align}  
		Since $\text{Preper}(\mathcal{D})$ is stable under the action of $\text{Gal}(\bar K/K)$, for all $z \in \text{Preper}(\mathcal{D})$
		\begin{align}
			\sigma(I(z))&=I(\sigma(z))=I(w(z)) \nonumber \\ 
			& \label{e6}\Rightarrow \sigma(I(z))=I(w(z))
		\end{align}
		From equations \eqref{e5} and \eqref{e6}, we have $w(I(z))=I(w(z))~\forall~z \in \text{Preper}(\mathcal{D})$. But since $\text{Preper}(\mathcal{D})$ is infinite, by Zariski dense property, we get $w \circ I=I \circ w$. This shows $I \in \mathcal{D}.$
	\end{proof}
	\begin{proof}[\protect{Alternative proof of Theorem~\ref{t3.4}}]
		By Theorem \ref{t4.2}, there exists an element $\sigma \in \text{Gal}(\bar{K}/K)$ and an stable power series $w_1$ in $\mathcal{D}_1$ such that 
		\begin{align*}
			\sigma(z)=w_1(z) \ \forall \ z \in \text{Preper}(\mathcal{D}_1).
		\end{align*} 
		The set $\mathcal{Z}:=\text{Preper}(\mathcal{D}_1) \cap \text{Preper}(\mathcal{D}_2)$ is stable under the action of $\text{Gal}(\bar K/K)$, and hence for all $z \in \mathcal{Z}$, we have $\sigma(z) \in \mathcal{Z}$. Therefore $w_1(z) \in \mathcal{Z}$ because $\sigma(z)=w_1(z)~\forall~z \in \text{Preper}(\mathcal{D}_1)$. Since $\mathcal{Z} \subset \text{Preper}(\mathcal{D}_2)$ is infinite, by the Lemma \ref{l4.3}, we get $w_1 \in \mathcal{D}_2$. This forces $\mathcal{D}_1=\mathcal{D}_2$.
	\end{proof}
	
	We understand the difficulty of finding a strong example satisfying the statement of the main Theorem \ref{t3.4}, nevertheless, we have produced the following two situations towards justification:
	\begin{exmp}
		We establish our argument rather contrapositively. We claim that there can not be two \enquote{different} stable $p$-adic dynamical systems $\mathcal{D}_1$ and $\mathcal{D}_2$ over $\mathcal{O}_K$ satisfying the statement of Theorem \ref{t3.4}. For, if $\mathcal{D}_1 \neq \mathcal{D}_2$ satisfies $\text{Preper}(\mathcal{D}_1) \cap \text{Preper}(\mathcal{D}_2)$=infinite. Then there exists two noninvertible series $f_1(x),~f_2(x)$ respectively $\mathcal{D}_1,~\mathcal{D}_2$ such that $f_1 \circ f_2 \neq f_2 \circ f_1$. But since $\text{Preper}(\mathcal{D}_1) \cap \text{Preper}(\mathcal{D}_2)$ is infinite, both $f_1-f_2$ vanishes on the infinite set $\text{Preper}(\mathcal{D}_1) \cap \text{Preper}(\mathcal{D}_2)$. Thus by Zariski dense property, we have $f_1=f_2$ and hence $f_1 \circ f_2=f_1^{\circ 2}=f_2 \circ f_1$, which is a contradiction. Thus our claim is established.   
	\end{exmp}
	\begin{exmp} \label{e4.3}
		Consider the noninvertible series $f_{F}(x)=3x+x^3$ over $\mathbb{Z}_3$. It is an endomorphism of a 1-dimensional Lubin-Tate formal group $F$ over $\mathbb{Z}_3$. Our idea is to recoordinatize the endomorphism $f_{F}$ and to form its condensation $\left(f_F\left(x^{\frac{1}{p-1}}\right)\right)^{p-1}$. Define a map $h(x)=x^{p-1}$, $p=3$. Note that  $h \circ f_{F}=f \circ h$, and hence $h$ is an isogeny from $f_{F}$ to $f$.
		Consider a stable $p$-adic dynamical system $\mathcal{D}_1$ over $\mathbb{Z}_3$ consisting of the noninvertible series $f(x):=\left(f_F\left(x^{\frac{1}{2}}\right)\right)^2=9x+6x^2+x^3$ and the invertible series $u(x)=4x+x^2$. It can be checked that $f \circ u=u \circ f$. Here $\Theta_1:=\{0,+\sqrt {-3}, -\sqrt {-3}\} \subset T(f_F)$ and $\Theta_2:= \{0,-3,-3\} \subset T(f)$ are set of zeros of $f_{F}$ and $f$ respectively. Also note, the isogeny $h$ takes $\Theta_1$ to $\Theta_2$ because $h(0)=0,~ h(\sqrt {-3})=3,~ h(-\sqrt{-3})=3$. Clearly $f(x)$ can not be an endomorphism of the formal group $F$ (not even any formal group) because it has repeated root. So the choice $\mathcal{D}_1$ is nontrivial with background formal group $F$ and compatible with respect to the statement of the Theorem \ref{t3.4}, in other word, the dynamical systems in our theorem exists.  
		
		The more difficult is to find another stable $p$-adic dynamical system $\mathcal{D}_2$ satisfying same criteria as $\mathcal{D}_1$. We earnestly hope that any example satisfying the statement to that of Theorem \ref{t3.1} would lead us to find $\mathcal{D}_2$. 
		However, we can create an easier $\mathcal{D}_2$ as follows: 
		
		 For, let us consider another Lubin-Tate formal group $G$ over $\mathbb{Z}_3$ with noninvertible endomorphism $g_G$ satisfying $g_G(x)\equiv 3^2x$ (mod degree 2) and $g_G(x) \equiv x^{3^2}$ (mod $3 \mathbb{Z}_3)$. Such a non-invertible endomorphism is $g_G(x)=9x+30x^3+27x^5+9x^7+{{x}^{9}}$ which commutes an invertible endomorphism $u_G(x)=5x+5x^2+x^5$ such that $g_G(G(x,y))=G(g_G(x),g_G(y))$ and $u_G(G(x,y))=G(u_G(x),u_G(y))$. We form condensation of $g_G(x)$ by the isogeny $h_2(x)=x^2$ respectively as $g(x)=81x+540x^2+1386x^3+1782x^4+1287x^5+546x^6+135x^7+18x^8+x^9$. We have checked that $h_2 \circ g_G=g \circ h_2$ and $h_2$ maps the zeros of $g_G$ into the zeros of $g$. Further, $g$ has double roots $-3$ and hence it can not be an endomorphism of any formal group. We found an invertible series $\tilde{u}(x)=25x+50x^2+35x^3+10x^4+x^5$ such that $g \circ \tilde{u}=\tilde{u} \circ g$. Thus we construct $\mathcal{D}_2$ as a stable $p$-adic dynamical system consisting of the invertible series $\tilde{u}$ and the noninvertible series $g$, whose background formal group is $G$.
		
		Finally, since $\text{Preper}(\mathcal{D}_1)$ and $\text{Preper}(\mathcal{D}_2)$ are independent of choices of stable series in $\mathcal{D}_1$ and $\mathcal{D}_2$ respectively, we can take $\text{Preper}(\mathcal{D}_1)=\text{Preper}(u)$ and $\text{Preper}(\mathcal{D}_2)=T(g)$. But, $u \in \mathcal{D}_1$ commutes with $g \in \mathcal{D}_2$ and hence $\text{Preper}(\mathcal{D}_1) \cap \text{Preper}(\mathcal{D}_2)$ is infinite. On the other hand, $f_F$ commutes with $u_G$ and hence $\text{Tors}(F) \cap \text{Tors}(G)$ is infinite, which implies $F=G$. This says $\mathcal{D}_1=\mathcal{D}_2$.   
	\end{exmp}
	\begin{rem}
		The existence of the nontrivial stable $p$-adic dynamical systems $\mathcal{D}_1$ and $\mathcal{D}_2$ in the above Example \ref{e4.3} supports the Conjecture \ref{c3.1}.
	\end{rem}
	\begin{rem}
		The Theorem \ref{t3.1} deals with the category of stable $p$-adic dynamical systems which are endomorphisms of formal groups while the main Theorem \ref{t3.4} of this paper deals with larger category of stable $p$-adic dynamical systems.  
	\end{rem}
	
	\textbf{Acknowledgement:} The authors would like to thank Jonathan Lubin, Laurent Berger for their insightful comments. The first author acknowledges \textit{The Council Of Scientific and Industrial Research}, Government of India, for the award of Senior Research Fellowship with File no.-09/025(0249)/2018-EMR-I.

\end{document}